\documentclass[preprint,12pt]{elsarticle}
\usepackage{amssymb}
\usepackage{amsthm}
\journal{Elsevier}
\usepackage[dvipsnames]{xcolor}

\usepackage{amsmath}
\usepackage{amsfonts}
\usepackage{amssymb}
\usepackage{xcolor}
\usepackage{enumerate}
\usepackage{lineno}
\usepackage[all]{xy}
\usepackage{tikz} \usetikzlibrary{calc} 

\newtheorem{definition}{Definition}[section]
\newtheorem{proposition}[definition]{Proposition}
\newtheorem{lemma}[definition]{Lemma}
\newtheorem{theorem}[definition]{Theorem}
\newtheorem{corollary}[definition]{Corollary}

\newdefinition{example}[definition]{Example}
\newdefinition{remark}[definition]{ Remark}
\newdefinition{problem}[definition]{ Problem}
\newdefinition{question}[definition]{Question}
\newdefinition{fact}[definition]{Fact}
\newproof{pot}{Proof}

\begin{document}

\begin{frontmatter}
\title{The hyperspaces $HS(p,X)$}

\author{Florencio~Corona--V\'azquez}
\ead{florencio.corona@unach.mx}

\author{Russell~Aar\'on~Qui\~nones--Estrella}
\ead{rusell.quinones@unach.mx}

\author{Javier~S\'anchez--Mart\'inez  \corref{cor1}}
\ead{jsanchezm@unach.mx}

\cortext[cor1]{Corresponding author}

\address{Universidad Aut\'onoma de Chiapas, Facultad de Ciencias en F\'isica y Matem\'aticas, Carretera Emiliano Zapata Km. 8, Rancho San Francisco, Ter\'an, C.P. 29050, Tuxtla Guti\'errez, Chiapas, M\'exico.}

\begin{abstract}
Let $X$ be a continuum and let $C(X)$ denote the hyperspace of subcontinua of $X$, endowed with the 
Hausdorff metric. For $p\in X$, define the hyperspace $C(p,X)=\{A\in C(X):p\in A\}$ as a subspace of $C(X)$. 
In this paper we introduced the quotient space $HS(p,X)=C(X)/C(p,X)$. We present some 
general properties of $HS(p,X)$ and we study the relationship between the continuum $X$ and the hyperspaces 
$C(X)$ and $HS(p,X)$.
\end{abstract}

\begin{keyword}
Continua  \sep hyperspaces \sep  quotient spaces.
\MSC Primary 
\sep 54B05 
\sep 54B20 
\sep 54F65 
\end{keyword}

\end{frontmatter}
\section{Introduction}

A \textit{continuum} is a nondegenerate compact connected metric space. From now on $X$ will denote a continuum, $2^{X}$ the corresponding  hyperspace of all nonempty closed subsets of $X$, $C_{n}(X)$ the hyperspace of at most $n$ connected components of $X$, when $n=1$, $C_{1}(X)$ will be denoted simply by $C(X)$, and $F_{n}(X)$ the hyperspace of all nonempty subsets of $X$ having at most $n$ points. All  hyperspaces above are considered with the Hausdorff metric (see \cite[p.~1]{Nadler(1978)}).
Given a point $p\in X$, $C(p,X)$ denotes the hyperspace of all subcontinua containing $p$, as a subspace of $C(X)$. It is well known that $C(p,X)$ is an AR continuum and consequently  locally connected (see \cite[Theorem 2, p. 220]{Eberhart(1978)}).  The hyperspaces $C(p,X)$ have been very useful to characterize some classes of continua, see for example \cite{Pellicer(2003)}, \cite{Pellicer(2005)}, \cite{Pellicer(2005b)}. In \citep{CESV(2018)} and \cite{CEST(2019)} the authors studied the topological structure of $C(p,X)$ in the case when $X$ is a finite graph.\\

On the other hand, in 1979 S. B. Nadler, Jr. introduced the \textit{hyperspace suspension of a continuum} $X$, $HS(X)$, as the quotient space $C(X)/F_{1}(X)$, \cite{Nadler(1979)}, which was subsequently studied in \cite{Escobedo(2004)}. In 2004 S. Mac\'ias, defines the \textit{$n$--fold hyperspace suspension of a continuum} $X$, $HS_{n}(X)$, as the quotient space $C_{n}(X)/F_{n}(X)$, \cite[\mbox{p.~127}]{Macias(2004)}. In a similar setting, other quotients of hyperspaces have been studied, for example $C_{n}(X)/F_{1}(X)$ in \cite{JCMacias(2008)}, $F_{n}(X)/F_{1}(X)$ in \cite{Barragan(2010)} and $F_{n}(X)/F_{m}(X)$ in \cite{CS(2013)}.\\

Let $p\in X$. In this paper we introduce  $HS(p,X)$ as the quotient space $C(X)/C(p,X)$.  The fact that $HS(p,X)$ is a continuum follows from \cite[3.14, p. 41]{Nadler(1992)}. We present some general properties of $HS(p,X)$ and study the relations between $X, C(X)$ and $HS(p,X)$ under concepts as dimension, unicoherence, locally connectedness,  aposyndesis and colocally connectedness.

\section{Definitions and preliminaries}
As customary, for $A\subset X$, $\textrm{cl}(A)$ and $\textrm{int}(A)$ denote the  closure and interior of $A$ in $X$ respectively.

By a \textit{finite graph} we mean a continuum $X$ which can be written as the union of finitely many arcs, any two of which are either disjoint or intersect only in one or both of their end points. A \textit{tree} is a finite graph without simple closed curves. Given a positive integer $n$,  a \textit{simple $n$-od} is a finite graph, denoted by $T_n$, which is the union of $n$ arcs emanating from a single point, $v$, and  otherwise disjoint from each another. The point $v$ is called the \textit{vertex} of the simple $n$--od. A simple $3$--od, $T_{3}$, will be called a \textit{simple triod}.  A $n$\textit{--cell}, $I^{n}$, is any space homeomorphic to $[0,1]^n$.  A \textit{simple closed curve} is a continuum homeomorphic to the unit circle $S^{1}=\{(x,y)\in\mathbb{R}^{2}:x^{2}+y^{2}=1\}$.

Given a finite graph $X$, $p\in X$ and a positive integer $n$, we say that $p$ \textit{is of order $n$ in $X$}, denoted by 
ord$(p,X)=n$, if $p$ has a closed neighborhood which is homeomorphic to a simple $n$--od having $p$ as the vertex. 
If ord$(p,X)=1$ the point $p$ is called an \textit{end point of $X$}. The set of all end points of $X$ will be denoted by $E(X)$. If ord$(p,X)=2$ the point $p$ is called an \textit{ordinary point of $X$}. The set of all ordinary points of $X$ will be denoted by $O(X)$. A point $p\in X$ is a \textit{ramification point of $X$} if ord$(p,X)\geq 3$. The set of all ramification points of $X$ will be denoted by $R(X)$.

Let $d$ denote the metric of $X$. Given $\varepsilon>0$ and $p\in X$ we define $B_{\varepsilon}(p)=\{x\in X: d(x,p)<\varepsilon\}$ and if $A\subset X$, we define:
$$N(\varepsilon,A)=\left\lbrace  x\in X:\textrm{there exists $y\in A$ such that $x\in B_{\varepsilon}(y)$}\right\rbrace .$$
If $A$ and $B$ are two closed subsets of $X$, then the \textit{Hausdorff distance
between $A$ and $B$} is defined by:
$$H(A,B)=\inf\{\varepsilon>0: A\subset N(\varepsilon,B)\textrm{ and $B\subset N(\varepsilon,A)$}\}.$$
It is well known that $H$ is a metric for $2^{X}$ (see \cite[Theorem 2.2, p. 11]{IllanesNadler(1999)}) and is called the \textit{Hausdorff metric}.

Given a finite collection $K_{1},\dots,K_{r}$ of subsets of $X$, $\left\langle K_{1},\dots,K_{r}\right\rangle$, denotes the following subset of $2^X$: $$\left\{A\in 2^{X}:A\subset \displaystyle\bigcup^{r}_{i=1}K_{i}, A\cap K_{i}\neq\emptyset\textrm{ for each }i\in\left\{1,\dots,r\right\}\right\}.$$ It is known that the family of all subsets of $2^{X}$ of the form $\left\langle K_{1},\dots,K_{r}\right\rangle$, where each $K_{i}$ is an open subset of $X$, forms a basis for a topology for $2^{X}$ (see \cite[Theorem 0.11, p.~9]{Nadler(1978)}) called the \textit{Vietoris Topology}. The Vietoris topology and the topology induced by the Hausdorff metric coincide (see \cite[Theorem 0.13, p.~10]{Nadler(1978)}).

A \textit{mapping} is a continuous function. 

The quotient mapping from $C(X)$ to $HS(p,X)$ is denoted by $\pi^X_p$, and $C^{X}_{p}$ denotes the point $\pi^X_p(C(p,X))$ in $HS(p,X)$. If there is no confusion, we will write down $[A]=\pi^{X}_{p}(A)$ for any $A\in C(X)-C(p,X)$.

\begin{remark}\label{complements}
By using the restriction of the projection $\pi^X_p$, it is clear that $C(X)-C(p,X)$ is 
homeomorphic to $HS(p,X)-\{C^{X}_{p}\}$.
\end{remark}

Given $A,B\in C(X)$ such that $A\subset B\neq A$, an \textit{order arc from $A$ to $B$} is a mapping  $\alpha:\left[0,1\right]\rightarrow C(X)$ such that $\alpha(0)=A$, $\alpha(1)=B$ and $s<t$ implies that $\alpha(s)\subset
\alpha(t)\neq \alpha(s)$. The existence of order arcs is standard stuff (cf. \cite[(1.8) and (1.11)]{Nadler(1978)}).

For a topological space $Z$, the \textit{cone over $Z$}, denoted by $\textrm{cone}(Z)$, is the quotient space $Z\times [0,1]/Z\times\{1\}$. 

\section{Examples}
In this section we present geometric models for $HS(p,X)$ for some continua, to give 
the reader an idea of the kind of spaces we will be concerned in this paper.
its 

\begin{example}\label{arco}
Let $X=[0,1]$. It is well known that the mapping $h:C(X)\to 
\mathbb{R}^{2}$ given by $h([a,b])=\left( \frac{a+b}{2},b-a\right) $ for each $[a,b]\in C(X)$ is an embedding. Also, $h(C(X))$ is the 
triangular $2$--cell, $T$, with vertices $(0,0)$, $(1,0)$ and $(\frac{1}{2},1)$. Let $p\in (0,1)$, by using the 
homeomorphism $h:C(X)\to T$ it is easy to see that  $HS(p,X)$ is homeomorphic to a space getting by gluing two 
2--cells by a point, which corresponds to $C^{X}_{p}$. 
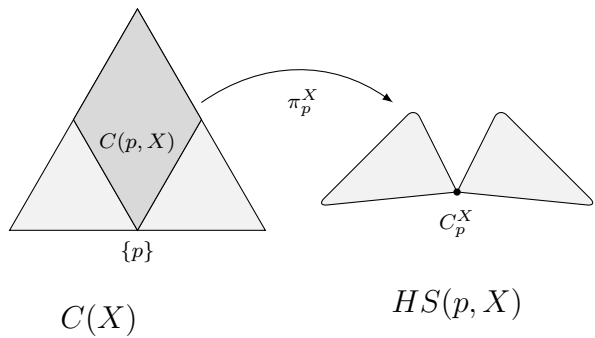
\[
\begin{tikzpicture}[scale=1.7]
   \begin{scope}[xshift=-2.5cm] 
    \coordinate (a) at (0,0); \coordinate (b) at (0:2cm); \coordinate (c) at (60:2cm);
    \draw[fill=gray!10] (a) -- (b) -- (c) -- cycle;
    \coordinate (v) at (0:1cm);
    \coordinate (medl) at (60:1cm); \coordinate (medr) at ($(v)+(60:1cm)$);
    \draw (0:1cm) -- (medl){}; \draw (v)-- (medr);
    \draw[fill=gray!30] (v) -- (medr) -- (c) -- (medl) -- cycle;
    \node[below] at (v){\scriptsize $\{p\}$};  \node[above=25pt] at (v){\scriptsize $C(p,X)$};
    \draw node at (0.7,-0.7){$C(X)$};
   \end{scope}
 \draw[-latex] (-1,1) to[out=35, in=145] (0.5,1); \draw node at (-0.2,1){\scriptsize $\pi _p^X$}; 
 \begin{scope}[xshift=1cm,yshift=0.3cm,scale=1.1]
  \draw[fill=gray!10] (0,0) {[rounded corners] -- (-0.3,0.6) -- (-1,-0.1)} -- (0,0) 
  {[rounded corners]-- (1,-0.1) -- (0.3,0.6)} -- cycle; 
  \draw node at (0,0){\tiny$\bullet$} node at (0,0)[below=3pt]{\scriptsize $C_p^X$};
  \draw node at (0,-0.8){$HS(p,X)$};
 \end{scope}   
\end{tikzpicture}
\]
In case $p=0$ or $p=1$ we have that each of $HS(0,X)$, $HS(1,X)$ and $C(X)$ are homeomorphic.
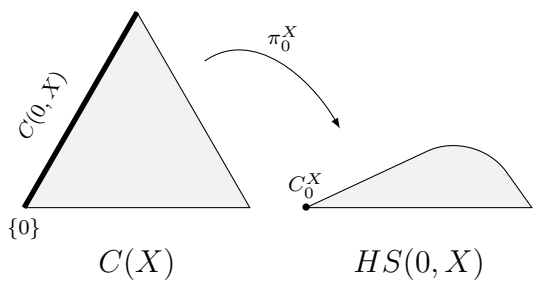
\[
\begin{tikzpicture}[scale=1.5]
 \begin{scope}[xshift=-1.3cm]
  \draw[fill=gray!10] (0,0) -- (2,0) -- (60: 2cm) -- cycle;
  \draw[line width=2pt] (0,0) -- (60:2cm);
  \draw[below] node at (0,0){\scriptsize $\{0\}$} node at (0,1)[rotate=60]{\scriptsize $C(0,X)$};
  \draw node at (1,-0.5){$C(X)$};
 \end{scope}
\draw node at (1,1.5){\scriptsize$\pi_0^X$} node at (1.2,0.2){\scriptsize $C_0^X$};
\draw[-latex] (0.3,1.3) to[out=35, in=120] (1.5,0.7); 
  \begin{scope}[xshift=1.2cm]
   \draw[fill=gray!10] {[rounded corners=0.7cm] (2,0) -- (1.5,0.7)} -- (0,0) --cycle;
   \draw node at (0,0){\tiny $\bullet$} node at (1,-0.5){$HS(0,X)$};
  \end{scope}
\end{tikzpicture}
\]
\end{example}

\begin{example}\label{circle}
Let $X=S^{1}$. By using the geometric model of $C(X)$ given in \cite[Example 5.2, p. 35]{IllanesNadler(1999)}, 
it is easy to see that $HS(p,X)$ is homeomorphic to $C(X)$ for each point $p\in X$.
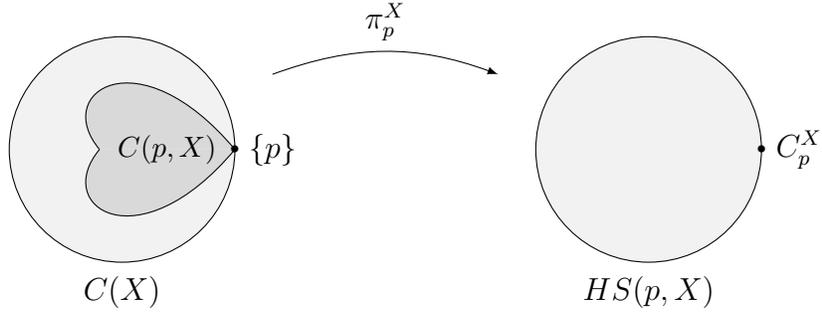
\[
\begin{tikzpicture}
   \draw[fill=gray!10] (0,0) circle (1.5cm);
   \draw[fill=gray!30] (-0.3,0) .. controls +(130:1cm) and +(130:2cm) .. (1.5,0);
   \draw[fill=gray!30] (-0.3,0) .. controls +(-130:1cm) and +(-130:2cm) .. (1.5,0);
  \draw[fill=gray!10] (7,0) circle (1.5cm); 
  \draw node at (8.5,0){\tiny $\bullet$};
 \draw[-latex] (2,1) to[out=20, in=160] (5,1);
\draw node at (0,-1.9) {$C(X)$}
      node at (0.6,0){\small $C(p,X)$} 
      node at (2,0){$\{p\}$} node at (1.5,0){\tiny $\bullet$};
\draw node at (7,-1.9) {$HS(p,X)$} node at (9,0){$C_p^X$} node at (3.5,1.7){$\pi_p^X$};
\end{tikzpicture}
\]
\end{example}

\begin{example}\label{triod}
Let $X$ be a simple n--od with vertex $v$. By \cite[Example 5.4, p. 41]{IllanesNadler(1999)}, $C(X)$ is 
homeomorphic to a $n$--cell, $I^{n}$, with $n$ 2--cells attached on it and sharing a common point.
In that geometric model $C(v,X)$ is represented as $I^{n}$. So, $HS(v,X)$ 
is homeomorphic to $n$ 2--cells glued in a unique point, $C^{X}_{v}$.
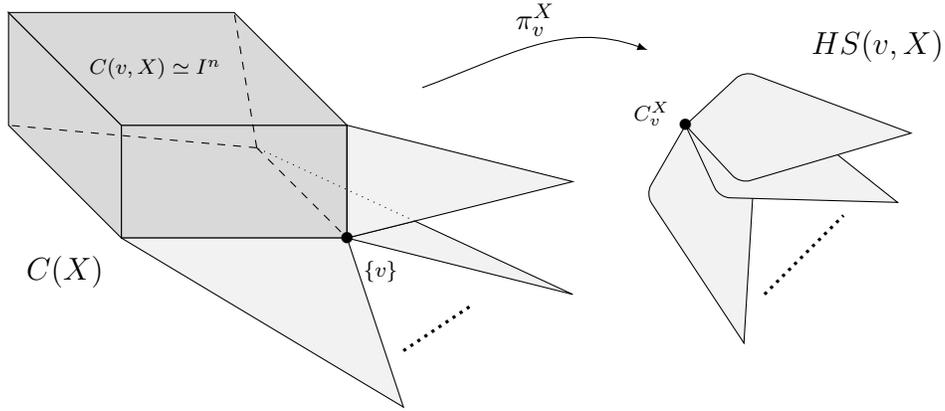
\[
\begin{tikzpicture}
 \begin{scope}[rotate=90,scale=1.5] %
  \draw[fill=gray!30] (0,0) --(1,0) -- (1,2) --(0,2) --cycle;
  \draw[fill=gray!30] (0,2) --++(1,1) --++(1,0) --(1,2) -- cycle;
  \draw[fill=gray!30] (1,0) --++(1,1) --++(0,2) --(1,2) -- cycle;
  \coordinate (v2) at (0.8,0.8);
  \foreach \x/\y in {2/1,1/3,0/0}{
      \draw[dashed] (v2) -- (\x,\y);
  }
  \draw[fill=gray!10] (0,0) -- (-1.5,-0.5) -- (0,2);
    \coordinate (a) at (intersection of 0.8,0.8-- -0.5,-2 and 0,0--0.5,-2);
  \draw[fill=gray!10] (0,0) -- (a) -- (-0.5,-2) -- cycle;
  \draw[fill=gray!10] (0,0) -- (1,0) -- (0.5,-2)--cycle;
  \draw[dotted] (a) -- (v2);
  \draw node at (1.5,1.7){\scriptsize$C(v,X)\simeq I^n$} 
	node at (0,0){\small $\bullet$} 
	node at (-0.3,-0.3){\scriptsize$\{v\}$}
	node at (-0.3,2.5){$C(X)$};
  \draw[very thick, dotted] (-1,-0.5) -- +(0.4,-0.6);
  \draw (0,0) -- (0,2) (0,0)--(1,0);
 \end{scope}
\draw[-latex] (1,2) to[out=20, in=160] (4,2.5);
  \draw node at (2.5,2.9){\small $\pi _v^X$};
  \begin{scope}[xshift=4.5cm, yshift=1.5cm,scale=1.5]
    \foreach \x in{-75,-20,-2}{
      \begin{scope}[rotate=\x]
	\draw[fill=gray!10] (0,0) {[rounded corners] -- (0.5,0.5)} -- (2,0) {[rounded corners]-- (0.5,-0.5)} 
--cycle;
      \end{scope}
    }
    \draw node at (0,0){\small $\bullet$} 
	  node at (-0.3,0.1){\scriptsize $C_v^X$}
	  nodeat (1.7,0.7){$HS(v,X)$};
    \draw[very thick, dotted] (0.7,-1.5) -- +(0.7,0.7);
  \end{scope}
\end{tikzpicture}
\]
\end{example}

\begin{example}\label{paleta}
Let $Y=\lbrace(x,y)\in\mathbb{R}^{2}:(x+1)^{2}+y^{2}=1\rbrace\cup \lbrace(x,0)\in\mathbb{R}^{2}:x\in 
[0,1]\rbrace$ and denote by $p$ the point $(0,0)$. By \cite[Example 5.3, p. 36]{IllanesNadler(1999)}, $C(X)$ 
is homeomorphic to a 3--cell, $T$, with two 2--cells attached on it as in the picture below (see also Figure 
6 of \cite[p. 37]{IllanesNadler(1999)}). The hyperspace $C(p,X)$ coincides with $T$ in that 
model. We get that $HS(p,X)$ is homeomorphic to the union of a two 2--cells gluing by a unique point, 
$C^{X}_{p}$. 
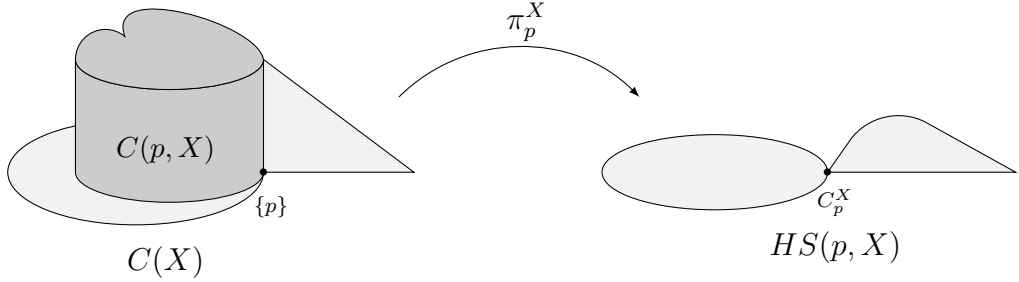
\[
\begin{tikzpicture}
  \begin{scope}[xshift=-3cm]
  \draw[fill=gray!10] (0,0) arc (360:110:1.7cm and 0.7cm); 
  \draw[fill=gray!40] (0,0) arc (360: 180: 1.25cm and 0.4cm) -- +(0,1.5) arc (180:360: 1.25cm and 0.4cm) 
--cycle;
  \draw[fill=gray!40] (-2.5,1.5) arc (180:50:0.4cm) --+(0,0) .. controls +(110:0.8cm) and +(100:0.5cm) .. 
(0,1.5) arc (360:180:1.25cm and 0.4cm);
  \draw[fill=gray!10] (0,0) -- (2,0) -- (0,1.5) --cycle;
\draw node at (0,0){\tiny $\bullet$} 
      node at (0.1,0)[below=5pt]{\scriptsize$\{p\}$}
      node at (-1.3,0.3){\small $C(p,X)$}
      node at (-1.3, -1.2){$C(X)$};
\end{scope}
\draw[-latex] (-1.2,1) to[out=45, in=135] (2,1);
\draw node at(0.5,2){$\pi_p^X$};
\begin{scope}[xshift=3cm]
  \draw[fill=gray!10] (0,0) circle (1.5cm and 0.5cm);
  \draw[fill=gray!10] (1.5,0) -- +(2.5,0) {[rounded corners=0.7cm] -- (2.2,1)} -- cycle;
  \draw node at (1.5,0){\tiny $\bullet$} 
	node at (1.6,-0.4){\scriptsize $C_p^X$}
	node at (1.6,-1){$HS(p,X)$};
\end{scope}
\end{tikzpicture} 
\]
\end{example}

\begin{example}\label{endpoints}
Let $X$ be a finite graph and $p\in E(X)$. If $X$ is not an arc, there exists a unique point $v\in R(X)$ such 
that if $U$ is the component of $X-\{v\}$ containing $p$, then $L:=U\cup \{v\}$ is an arc with end points 
$p$ and $v$. Set $G=\textrm{cl}(X-U)$, note that $C(X)=C(L)\cup C(G)\cup C(v,X)$ 
and $C(p,X)\cap C(G)=\emptyset$. Therefore, $C(G)$ and $\pi^{X}_{p}(C(G))$ are homeomorphic. By Example 
\ref{arco}, $\pi^{X}_{p}(C(L))$ is a 2--cell which is homeomorphic to $C(L)$. On the other hand, note that $C(v,X)$ is 
homeomorphic to $C(v,G)\times [0,1]$, in fact, if $f:[0,1]\to L$ is a homeomorphism such that 
$f(0)=v$ and $f(1)=p$, then the mapping $h:C(v,X)\to C(v,G)\times [0,1]$ given by 
$$h(A)=(A\cap G,f^{-1}(A\cap L)), \textrm{for each $A\in C(v,X)$},$$ is a homeomorphism. We 
can see that the homeomorphism $h$ send $C(p,X)\cap C(v,X)$ into $C(v,G)\times \{1\}$. Therefore, 
$\pi^{X}_{p}(C(v,X))$ is homeomorphic to $\textrm{cone}(C(v,G))$, moreover, the vertex of 
$\textrm{cone}(C(v,G))$ is $C^{X}_{p}$. 

Consequently, note that $HS(p,X)$ is homeomorphic to the union of the 2--cell, $\pi^{X}_{p}(C(L))$, the space homeomorphic 
to the $\textrm{cone}(C(v,G))$, $\pi^{X}_{p}(C(v,X))$, and the subspace 
$\pi^{X}_{p}(C(G))$ which is homeomorphic to $C(G)$, following the next 
intersections:
\begin{itemize}
\item $\pi^{X}_{p}(C(L))\cap \pi^{X}_{p}(C(v,X))=C(v,L)$, which is an arc;
\item $\pi^{X}_{p}(C(L))\cap \pi^{X}_{p}(C(G))=\{C^{X}_{p}\}$;
\item $\pi^{X}_{p}(C(G))\cap \pi^{X}_{p}(C(v,X))=\pi^{X}_{p}(C(v,G))$, which is homeomorphic to 
$C(v,G)$ because $C(v,G)\cap C(p,X)=\emptyset$. 
\end{itemize}

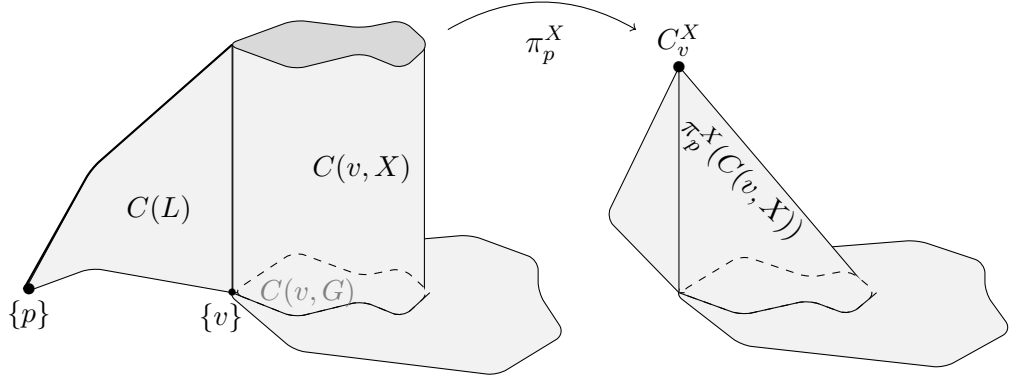
\[
\begin{tikzpicture}
\begin{scope}[rounded corners, xscale=1.8, yscale=1.1,xshift=-3.3cm]
\draw[fill=gray!10] (0,0) -- (0.5,-0.7) -- (2,-1) -- (2.5, -0.5) -- (2.3,0) -- 
      (2.3,0.4) -- (1.8,0.7) -- (0.4,0.5) -- cycle;
\draw[fill=gray!10,sharp corners] (0.045,0) rectangle (1.46,3);
\draw[fill=gray!10,dashed] (0.04,0) -- (0.5,-0.3) -- (1,-0.1) -- (1.3,-0.3) -- (1.5,0) --
      (1.3,0.3) -- (1,0.2) -- (0.5,0.5) -- cycle;
\draw[fill=gray!30,yshift=3cm] (0,0) -- (0.5,-0.3) -- (1,-0.1) -- (1.3,-0.3) -- (1.5,0) --
      (1.3,0.3) -- (1,0.2) -- (0.5,0.3) -- cycle;
\draw (0.04,0) -- (0.5,-0.3) -- (1,-0.1) -- (1.3,-0.3) -- (1.5,0);
\draw[fill=gray!10] (0.04,0) -- (-1,0.3) -- (-1.5,0) -- (-1,1.5) -- (0.04,3);
\draw (0.04,0) -- (0.04,3);
\draw[thick] (-1.5,0.04) -- (-1,1.5) -- (0.04,3);
\draw node at (-0.5,1) {\small $C(L)$}
      node at (1,1.5) {\small $C(v,X)$}
      node at (0.04,0){\tiny $\bullet$} node at (-0.04,-0.3){\small $\{v\}$}
      node at (-1.46,0.04){\small $\bullet$} node at (-1.46,0.04)[below]{\small $\{p\}$}
      node at (0.6,0){\color{gray} \small$C(v,G)$};
\end{scope}
\draw[->](-3,3.5) to[out=30,in=150](-0.5,3.5);
\draw node at (-1.7,3.3){$\pi_p^X$};
\begin{scope}[rounded corners, xscale=1.8]
\draw[fill=gray!10] (0,0) -- (0.5,-0.7) -- (2,-1) -- (2.5, -0.5) -- (2.3,0) -- 
      (2.3,0.4) -- (1.8,0.7) -- (0.4,0.5) -- cycle;
\draw[fill=gray!10,sharp corners]  (0.045,0) -- (1.46,0) -- (0.045,3);
\draw[fill=gray!10,dashed] (0.04,0) -- (0.5,-0.3) -- (1,-0.1) -- (1.3,-0.3) -- (1.5,0) --
      (1.3,0.3) -- (1,0.2) -- (0.5,0.5) -- cycle;
\draw (0.04,0) -- (0.5,-0.3) -- (1,-0.1) -- (1.3,-0.3) -- (1.5,0);
\draw[fill=gray!10] (0.04,0) -- (-0.5,1) -- (0.04,3);
\draw (0.04,0) -- (0.04,3);
\draw 
      node at (0.5,1.5) [rotate=-45]{\small $\pi _p^X(C(v,X))$}
      node at (0.04,3){\small $\bullet$} node at (0.04,3)[above]{\small $C_v^X$};
      \end{scope}
\end{tikzpicture}
\]
%
\end{example}

\begin{example}
Let $X$ be a simple $n$--od with vertex $v$ and $p\in E(X)$. We denoted by $L$ the unique arc in $X$ with end points $p$ and $v$ and set $G=\textrm{cl}(X-L)$. Since $C(v,X)$ is a $n$--cell which is homeomorphic to both $C(v,G)\times [0,1]$ and $\textrm{cone}( C(v,G))$, using the Example \ref{endpoints} we can see that $HS(p,X)$ is homeomorphic to $C(X)$.    
\end{example}

\section{Dimension of $HS(p,X)$}

In this section, \textit{dimension} means inductive dimension as defined in \cite[(0.44), p. 21]{Nadler(1978)}. The symbols $\dim(X)$ and $\dim_{p}(X)$ will be used to denote the dimension of the space $X$ and the dimension of $p$ in $X$, respectively.

\begin{lemma}\label{dimensionleq}
For each point $p\in X$, it holds that $\dim (HS(p,X))\leq \dim(C(X))$.
\end{lemma}

\begin{proof}
By Remark \ref{complements}, $HS(p,X)-\{C^{X}_{p}\}$ is homeomorphic to $C(X)-C(p,X)$. By \cite[Corollary 2, p. 32 and Theorem III.1, p. 26]{HurewiczWallman(1948)}, we have that 
\begin{eqnarray*}
\dim (HS(p,X)) & = & \dim (HS(p,X)-\{C^{X}_{p}\}) \\
& = & \dim(C(X)-C(p,X)) \\
& \leq & \dim (C(X)).
\end{eqnarray*}
\end{proof}

\begin{proposition}
Let $p\in X$. It holds that $\dim C(X)<\infty$ if and only if $\dim HS(p,X)<\infty$ and $\dim C(p,X)<\infty$. 
\end{proposition}

\begin{proof}
If $\dim HS(p,X)<\infty$,  then $\dim (C(X)-C(p,X))<\infty$ because of \cite[Theorem III.1, p. 26]{HurewiczWallman(1948)} and because $C(X)-C(p,X)$ is homeomorphic to $HS(p,X)-\{C^{X}_{p}\}$. Since $C(X)=(C(X)-C(p,X))\cup C(p,X)$, we conclude, by using \cite[B), p. 28]{HurewiczWallman(1948)}, that $\dim C(X)<\infty$. 

The converse follows directly from \cite[Theorem III.1, p. 26]{HurewiczWallman(1948)}  and  Lemma \ref{dimensionleq}. 
  
\end{proof}

The following result is a particular case of \cite[Theorem 2.4, p. 791]{Martinez(2006)}.

\begin{theorem}\label{veronica}
Let $X$ be a finite graph and $A\in C(X)$. Then:
\begin{enumerate}
\item $\dim_{A}(C(X))=2+\displaystyle\sum_{r\in R(X)\cap A}(\textrm{ord}(r,X)-2)$ and
\item $\dim_{X}(C(X))=\dim (C(X)).$
\end{enumerate}
\end{theorem}

Since $HS(p,X)-\{C^{X}_{p}\}$ is homeomorphic to the open set $C(X)-C(p,X)$ of $C(X)$, we have, as a consequence of the previous theorem, the following result.

\begin{corollary}\label{dimension}
Let $X$ be a finite graph and $p\in X$. If $A\in C(X)-C(p,X)$, then $$\dim_{[A]}(HS(p,X))=2+\displaystyle\sum_{r\in R(X)\cap A}(\textrm{ord}(r,X)-2).$$ 
\end{corollary}

Note that, if $X$ is a finite graph and $p\in R(X)$, then for each $A\in C(X)-C(p,X)$, it holds that $$\displaystyle\sum_{r\in R(X)\cap A}(\textrm{ord}(r,X)-2)<\displaystyle\sum_{r\in R(X)}(\textrm{ord}(r,X)-2).$$ 
From the above inequality, Theorem \ref{veronica} and Corollary \ref{dimensionleq}, we obtain the following.

\begin{corollary} \label{dimensionramification}
Let $X$ by a finite graph. Then, for each $p\in R(X)$ it holds that $\dim(HS(p,X))<\dim C(X)$.
\end{corollary}

A \textit{cut point} in $X$ means a point $p\in X$ such that $X-\{p\}$ is not connected, otherwise we say that $p$ {\it does not cut} $X$.

For a finite graph $X$, $R(X)=\emptyset$ if and only if $X$ is an arc or a simple closed curve (see \cite[Theorem 70.1, p. 337]{IllanesNadler(1999)}). Hence, by using examples \ref{arco} and \ref{circle}, and Corollary \ref{dimensionramification}, we get the following result.

\begin{theorem}
Let $X$ be a finite graph. Then:
\begin{enumerate}
\item $\dim(HS(p,X))=\dim(C(X))$ for each $p\in X$ if and only if $X$ is an arc or a simple closed curve.
\item $HS(p,X)$ is homeomorphic to $C(X)$ for each $p\in X$, if and only if $X$ is a simple closed curve.
\item $\dim(HS(p,X))=\dim(C(X))$ for each $p\in X$ and there exists $q\in X$ such that $HS(q,X)$ contains cut points, if and only if $X$ is an arc.
\end{enumerate}
\end{theorem}

\section{Basic properties of $HS(p,X)$}

Given a nonempty set $U\subset X$, let $C(U):=\{A\in C(X): A\subset U\}$. It is known that, if $U$ is open, then $C(U)$ is open (see \cite[1.1, p. 3]{IllanesNadler(1999)}). 

\begin{lemma}\label{Cconnected}
If $U$ is connected then $C(U)$ is connected. 
\end{lemma}

\begin{proof}
Since  $C(\emptyset)=\emptyset$ is connected, we can suppose that $U\neq \emptyset$. Let $A\in C(U)$ and take an arbitrary point $a\in A$. Let $\alpha_{A}:[0,1]\to C(X)$ be an ordered arc from $\{a\}$ to $A$. It follows from this that $C(U)=\displaystyle\bigcup_{A\in C(U)}\alpha_{A}([0,1])\cup F_{1}(U)$. We conclude that $C(U)$ is connected. 
\end{proof}

\begin{proposition}\label{cutpoints}
A point $p\in X$ is a cut point of $X$ if and only if $C^{X}_{p}$ is a cut point of $HS(p,X)$. 
\end{proposition}

\begin{proof}
First, suppose that $X-\{p\}=A\cup B$, where $A$ and $B$ are two nonempty disjoint open subsets of $X$. In this case, by Lemma \ref{Cconnected}, $C(X)-C(p,X)$ is not connected because it is the union of the open nonempty sets $C(A)$ and $C(B)$ which are disjoint. By Remark \ref{complements}, we conclude that $HS(p,X)-\{C^{X}_{p}\}$ is not connected.\\
If $X-\{p\}$ is connected, then $C(X-\{p\})=C(X)-C(p,X)$ is connected. Again, by Remark \ref{complements}, this implies that $HS(p,X)-\{C^{X}_{p}\}$ is connected. Therefore, $HS(p,X)-\{C^{X}_{p}\}$ not connected implies $X-\{p\}$ not connected.
\end{proof}

Recall that any $C(X)$ does not have cut points (see \cite[Exercise 6.8, p. 100]{Nadler(1992)}). From this and  Proposition \ref{cutpoints} we get the following.

\begin{corollary}
Let $p\in X$. If $C(X)$ is homeomorphic to $HS(p,X)$, then $X-\{p\}$ is connected.
\end{corollary}

We say that $X$ is \textit{uniformly pathwise connected} provided that it is the continuous image of the cone over the Cantor set (i.e. Cantor fan) \cite[3.5]{Kuperberg(1975)}. By using that $\pi^{X}_{p}$ is a surjective mapping and that $C(X)$ is a continuous image of the Cantor fan (see e.g. \cite[17.10]{IllanesNadler(1999)}), we have the following result.

\begin{theorem}
$HS(p,X)$ is uniformly pathwise connected for any $p\in X$.
\end{theorem}


Let $f:X\to Y$ be a mapping between continua. The \textit{induced} mapping $C(f):C(X)\to C(Y)$ is given by $C(f)(A)=f(A)$ (see \cite{Hosokawa(1997)}). Observe that, for each $p\in X$, we have a unique mapping $HS(p,f):HS(p,X)\to HS(f(p),Y)$ such that $\pi^Y_{f(p)}\circ C(f)=HS(p,f)\circ \pi^X_p$, i.e., the following diagram commutes (see \cite[Theorem 4.3, p.~126]{Dugundji(1966)}).

\[
\xymatrix{& C(X)\ar[rrr]^{C(f)} \ar[dd]_{\pi^X_p}& & &  C(Y)\ar[dd]^{\pi^Y_{f(p)}}\\
	& & & & &\\
	& HS(p,X)\ar[rrr]_{HS(p,f)}& & &  HS(f(p),Y)}
\]




\begin{theorem}[Functorial properties]
Let $f:X\to Y$ and $h:Y\to Z$ be mappings between continua and $p\in X$. It holds
\begin{enumerate}
 \item $HS(f(p),h)\circ HS(p,f)=HS(p,h\circ f)$.
 \item If $i:X\to X$ is the identity mapping, then $HS(p,i)$ is the identity homeomorphism.
\end{enumerate}
\end{theorem}

\begin{proof}
For each $[A]\in HS(p,X)$, 
\begin{eqnarray*}
HS(p,h\circ f)([A]) &=& \pi^{Z}_{h(f(p))}(h(f(A))) \\
& = & HS(f(p),h)(\pi^{Y}_{f(p)}(f(A))) \\
& = & HS(f(p),h)(HS(p,f)([A])),
\end{eqnarray*}
this conclude 1. \\
For the second part, observe that 
$$HS(p,i)([A])=\pi^{X}_{p}(i(A))=\pi^{X}_{p}(A)=[A].$$
\end{proof}

As an immediate consequence of the previous theorem, we get the following. 

\begin{corollary}\label{homeomorphism}
If $f:X\to Y$ is a homeomorphism sending $p\in X$ to $q\in Y$, then $HS(p,f):HS(p,X)\to HS(q,Y)$ is a homeomorphism.
\end{corollary}




The case of our interest lay in the class of trees.

\begin{corollary}\label{homeomorphics}
Let $X$ and $Y$ two trees, $p\in X$ and $q\in Y$. If $C(p,X)$ is homeomorphic to $C(q,Y)$ then $HS(p,X)$ is homeomorphic to $HS(q,Y)$.
\end{corollary}

\begin{proof}
By \cite[Theorem 4.14]{CEST(2019)}, we have that there exists a homeomorphism $h:X\to Y$ sending $p$ to 
$q$, now apply Corollary \ref{homeomorphism} to conclude.
\end{proof}

\begin{example}
The converse of the previous corollary is not true in general, for example, the graphs of examples 
\ref{triod} (case $n=3$) and \ref{paleta} satisfy that $HS(v,X)$ is not 
homeomorphic to $HS(p,Y)$ even though $C(v,X)$ and $C(p,Y)$ are homeomorphic.

\end{example}

Given a positive integer $n$, we denote by $\mathcal{U}_{n}(X)$ the set of all points $x\in X$ 
such that $\dim_{x}(X)\leq n$. In the case when $X$ is a finite graph, it is known (see \cite[Theorem 2.4, p. 
791]{Martinez(2006)}) that 
\[
\mathcal{U}_{n}(X)=\left\lbrace A\in C(X):n\geq 2+\sum_{p\in R(X)\cap 
A}(\textrm{ord}(p,X)-2)\right\rbrace ,
\]
consequently $\mathcal{U}_{2}(X)=\{A\in C(X):A\cap R(X)=\emptyset\}$.\\

For a topological space $Z$ we denote by $\pi_{0}(Z)$, as customary, the set of all connected components of 
$Z$ and by $\vert Z\vert$ its cardinality.

\begin{lemma}\label{dimensiones}
Let $X$ be a finite graph and $p\in O(X)$. If $R(X)\neq\emptyset$ then 
\[
 \vert\pi_{0}(\mathcal{U}_{2}(C(X))-C(p,X))\vert=\vert\pi_{0}(\mathcal{U}_{2}(HS(p,X))\vert.
\]
\end{lemma}

\begin{proof}
Since $C(X)-C(p,X)$ and $HS(p,X)-\{C^{X}_{p}\}$ are homeomorphic, 
\[
 \vert\pi_{0}(\mathcal{U}_{2}(C(X))-C(p,X))\vert=\vert\pi_{0}(\mathcal{U}_{2}(HS(p,X)-\{C^{X}_{p}\})\vert.
\]
On the other hand, it is easy to see that, since $R(X)\neq \emptyset$ and $C^{X}_{p}\notin 
\mathcal{U}_{2}(HS(p,X))$, it holds $\mathcal{U}_{2}(HS(p,X)-\{C^{X}_{p}\})=\mathcal{U}_{2}(HS(p,X))$, which 
conclude this lemma.
\end{proof}

\begin{lemma}\label{domensiones}
Let $X$ be a finite graph and $p\in O(X)$. If $R(X)\neq \emptyset$, then $$\vert\pi_{0}\left( \mathcal{U}_{2}(C(X)-C(p,X))\right) \vert>\vert\pi_{0}(\mathcal{U}_{2}(C(X))\vert.$$
\end{lemma}

\begin{proof}
Let $A$ be the edge of $X$ containing $p$. Since $R(X)\neq\emptyset$, $X\neq A$, moreover $A\cap 
R(X)\neq\emptyset$. Consider $q\in R(X)\cap A$. Let $A_{1}$ and $A_{2}$ represent the unique subarcs of $A$ such 
that $A=A_{1}\cup A_{2}$ and, $A_{1}\cap A_{2}=\{p\}$ if $A$ is an arc, or, $A_{1}\cap A_{2}=\{p,q\}$ if $A$ 
is a simple closed curve. Now, set $W:=\{B\in \mathcal{U}_{2}(C(X)):B\subset X-A\}$. Thus $\mathcal{U}_{2}(C(X))=W\cup \mathcal{U}_{2}(C(A))$ and
\begin{center}
 $\mathcal{U}_{2}(C(X)-C(p,X))=W\cup \mathcal{U}_{2}(C(A_{1})-C(p,A_{1}))\cup \mathcal{U}_{2}(C(A_{2})-C(p,A_{2})).$
\end{center}
From this, it is easy to see that $\mathcal{U}_{2}(C(A_{1})-C(p,A_{1}))$ and $\mathcal{U}_{2}(C(A_{2})-C(p,A_{2}))$ are components of $\mathcal{U}_{2}(C(X)-C(p,X))$ and that $\mathcal{U}_{2}(C(A))$ is a component of $\mathcal{U}_{2}(C(X))$. Then, $$\vert\pi_{0}\left( \mathcal{U}_{2}(C(X)-C(p,X))\right) \vert=\vert\pi_{0}(\mathcal{U}_{2}(C(X))\vert+1,$$ which conclude the proof.
\end{proof}

By using lemmas \ref{dimensiones} and \ref{domensiones}, we get the next result.

\begin{corollary}\label{desigualdad}
Let $X$ be a finite graph and $p\in O(X)$. If $R(X)\neq\emptyset$, then $$\vert\pi_{0}(\mathcal{U}_{2}(HS(p,X))\vert>\vert\pi_{0}(\mathcal{U}_{2}(C(X))\vert,$$ in particular $HS(p,X)$ is not homeomorphic to $C(X)$.
\end{corollary}

\begin{theorem}\label{ordirarypoints}
Let $X$ be a finite graph and $p\in O(X)$. Then $C(X)$ is homeomorphic to $HS(p,X)$ if and only if $X$ is a simple closed curve.
\end{theorem}

\begin{proof}
By Example \ref{circle}, it is sufficient to show that if $C(X)$ is homeomorphic to $HS(p,X)$ then $X$ is a 
simple closed curve. Suppose on the contrary, that $X$ is not a simple closed curve and that there exists a 
homeomorphism $h:C(X)\to HS(p,X)$. From Example \ref{arco} and \cite[Theorem 70.1, p. 
337]{IllanesNadler(1999)}, we have that $R(X)\neq \emptyset$, but this is impossible by Corollary 
\ref{desigualdad}.
\end{proof}

\begin{example}\label{finales}
Example \ref{arco} shows a tree, the arc $[0,1]$, which satisfies that $0\in E([0,1])$ and $HS(0,[0,1])$ is 
homeomorphic to $C([0,1])$. We give now an example of a tree $X$, with a point $p\in E(X)$ such that 
$HS(p,X)$ is not homeomorphic to $C(X)$. Consider the following subspace of the Euclidean plane, $X=(\{0\}\times [-1,1])\cup ([0,1]\times \{0\})\cup (\{1\}\times 
[-1,1])\cup (\{\frac{1}{2}\}\times [-1,0])$. Put $p=(0,1)\in E(X)$. For each $n\in\mathbb{N}$ let:
\begin{itemize}
\item $A_{n}=(\{0\}\times [0,1-\frac{1}{n}])\cup ([0,1]\times \{0\})$, and
\item $B_{n}=\{0\}\times [0,1-\frac{1}{n}]$.
\end{itemize}
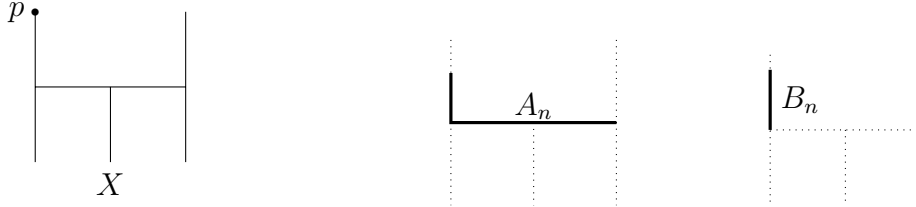
\[
\begin{tikzpicture}
\draw node at (-1,1) {\tiny $\bullet$} node at (-1,1) [left]{$p$} node at (0,-1.3){$X$};
\draw (-1,1) -- (-1,-1) (1,1) -- (1,-1) (-1,0)--(1,0) (0,0) -- (0,-1);
\end{tikzpicture} 
\hspace{3.5cm}
\begin{tikzpicture}[scale=1.1]
\draw[dotted] (-1,1) -- (-1,-1) (1,1) -- (1,-1) (-1,0)--(1,0) (0,0) -- (0,-1);
\draw[very thick] (-1,0.6) -- (-1,0) -- (1,0);
\draw node at (0,0.2){$A_n$};
\end{tikzpicture} 
\hspace{2cm}
\begin{tikzpicture}
\draw[dotted] (-1,1) -- (-1,-1) (1,1) -- (1,-1) (-1,0)--(1,0) (0,0) -- (0,-1);
\draw[very thick] (-1,0.8) -- (-1,0);
\draw node at (-0.6,0.4){$B_n$};
\end{tikzpicture} 
\]
 
By Corollary \ref{dimension}, $\textrm{dim}_{[A_{n}]}(HS(p,X))=5$ and $\textrm{dim}_{[B_{n}]}(HS(p,X))=3$. 
The sequences $\{[A_{n}]\}_{n\in\mathbb{N}}$ and $\{[B_{n}]\}_{n\in\mathbb{N}}$ converge to $C^{X}_{p}$ in 
$HS(p,X)$. 

If we suppose that $h:HS(p,X)\to C(X)$ is a homeomorphism, there should 
exists two sequences in $C(X)$, say $\{D_{n}\}_{n\in\mathbb{N}}$ and $\{F_{n}\}_{n\in\mathbb{N}}$, converging 
to the point $H:=h(C^{X}_{p})$ in such a way that $\textrm{dim}_{D_{n}}C(X)=5$ and $\textrm{dim}_{F_{n}}C(X)=3$. 
Therefore, for each $n$, 
$[0,1]\times \{0\}\subset D_{n}$ 
 and then $[0,1]\times \{0\}\subset H$, so there exists $m\in\mathbb{N}$ such that $(\frac{1}{2},0)\in F_{m}$, thus $\textrm{dim}_{F_{n}}C(X)>3$ which is impossible, this conclude that the homeomorphism $h$ can not exist.
  
\end{example}

\begin{example}
By using Corollary \ref{dimensionramification} we can see that if $X$ is a finite graph and $p\in R(X)$, then 
$C(X)$ is not homeomorphic to $HS(p,X)$.
\end{example}

Remember that an \textit{$n$--od} is a continuum $X$ such that there exists $A\in C(X)$ satisfying that $X-A$ 
has at lest $n$ different components.

\begin{theorem}
If $\dim(X)=1$, then $X$ contains an $n$--od if and only if there exists $p\in X$ such that $HS(p,X)$ contains an $n$--cell.
\end{theorem}

\begin{proof}
If $X$ contains an $n$--od, say $Y\in C(X)$, then there exists $A\in C(Y)$ such that $Y-A$ has $n$ different 
components, $A_{1},\cdots, A_{n}$. By \cite[Corollary 5.9, p. 75]{Nadler(1992)}, for each 
$i\in\{1,\cdots,n\}$, $A\cup A_{i}\in C(Y)$. Let $\alpha_{i}:[0,1]\to C(Y)$ be an ordered arc from $A$ to 
$A\cup A_{i}$. Take $p\in (A\cup A_{1})-\alpha_{1}(\frac{1}{2})$. It is clear that $\alpha:[0,1]^{n}\to C(X)$ 
given by $\alpha(t_{1},\cdots,t_{n})=\alpha_{1}(\frac{t_{1}}{2})\cup \bigcup_{i=2}^{n}\alpha_{i}(t_{i})$ is an 
embedding and $\alpha([0,1]^{n})\subset C(X)-C(p,X)$. Therefore, $\pi^{X}_{p}(\alpha([0,1]^{n}))$ is an 
$n$--cell contained in $HS(p,X)$.

If $HS(p,X)$ contains an $n$--cell, then $HS(p,X)-\{C^{X}_{p}\}$ contains an $n$--cell, therefore 
$C(X)-C(p,X)$ contains an $n$--cell. Now just apply \cite[70.1]{IllanesNadler(1999)} to get that $X$ 
contains an $n$--od.
\end{proof}

\begin{remark}\label{ncells}
Observe that in the previous result the condition of being 1--dimensional is not necessary to show that if $X$ 
contains any $n$--od then $HS(p,X)$ contains $n$--cells.
\end{remark}

\section{Unicoherence}

Remember that $X$ is called \textit{unicoherent} provided that, whenever $A$ and $B$ are closed, connected 
subsets of $X$ such that $X=A\cup B$, then $A\cap B$ is connected. 

A mapping $f:X\to Y$ between continua is said to be \textit{monotone} if for each $p\in Y$, $f^{-1}(p)\in C(X)$.

\begin{theorem}\label{unicoherent}
For each $p\in X$, $HS(p,X)$ is unicoherent.
\end{theorem}

\begin{proof}
It is well known that $C(X)$ is unicoherent (see \cite[19.8, p. 159]{IllanesNadler(1999)}). Since $\pi^X_p$ is monotone, $HS(p,X)$ is unicoherent (see \cite[1.21, p. 138]{Whyburn(1942)}).
\end{proof}

Recall that $X$ has \textit{property (b)} if, for each mapping $f:X\to S^{1}$ , there exists a mapping $g:X\to \mathbb{R}$ such that $f(x)=(\textrm{exp}\circ g)(x)$
for all $x\in X$ where $\textrm{exp}$ is the exponential mapping into the complex plane.

In general, if $X$ is a connected, normal topological space and $X$ has property (b), then $X$ is unicoherent (see \cite[Theorem 2, p. 69]{Eilemberg(1936)}). Moreover, it is well known that, if $X$ is a locally connected continuum, then $X$ has property (b) if and only if $X$ is unicoherent (see \cite[Theorem 3, p. 70]{Eilemberg(1936)}).

\begin{theorem}\label{propertyb}
$HS(p,X)$ has property (b) for any $p\in X$.
\end{theorem}

\begin{proof}
By \cite[Theorem 3]{Nadler(1970)}, $C(X)$ has property (b). The monotone image of a continuum satisfying property (b) has also the property (b) (see \cite[Theorem 2, p. 434]{Kuratowski(1978)}). Since $\pi^{X}_{p}$ is monotone, the result follows.
\end{proof}

Next, we remember some classical concepts before we go further in our study. $X$ is called 
\textit{hereditarily unicoherent} if for each $A\in C(X)$, $A$ is unicoherent. A \textit{dendroid} is an 
arcwise connected and hereditarily unicoherent continuum. A \textit{fan} is a dendroid with exactly one 
ramification point (i.e., with only one point which is the common part of three otherwise disjoint arcs). To 
see more information about fans see \cite{Charatonik(1967)}. The unique ramification point of a fan $F$ is 
called the \textit{top} of $F$ and is denoted by $v$. By an \textit{end point} of a fan $F$ we mean a point 
$e$ of $X$ that is a nonseparating point of any arc in $F$ that contains $e$; $E(F)$ denotes the set of all 
end points of a fan $F$. A \textit{leg} of a fan $F$ is the unique arc in $F$ from $v$ to some end point of 
$F$. For each $e\in E(F)$, denote by $L_{e}$ the leg with end points $e$ and $v$. 

\begin{example}
Let $X$ be a fan. Note that, $C(X)=C(v,X)\cup \displaystyle\bigcup_{e\in E(X)}C(L_{e})$. It is easy to see that $C(v,X)\cap C(L_{e})=C(v,L_{e})$, $C(X)-C(v,X)$ is homeomorphic to $HS(v,X)-\{C^{X}_{v}\}$ and for each $e\in E(X)$, $HS(v,L_{e})$ is homeomorphic to $C(L_{e})$. Consequently $HS(v,X)$ is homeomorphic to $\displaystyle\bigcup_{e\in E(X)}C(L_{e})$.
\end{example}

\section{Locally connectedness}
We begin this section recalling that $X$ has the \textit{property of Kelley} provided that for each 
$\varepsilon>0$, there exists $\delta>0$ such that if $p$ and $q$ are two points of $X$ such that 
$d(p,q)<\delta$ and if $A$ is a subcontinuum of $X$ such that $p\in A$, then there exists a subcontinuum $B$ 
of $X$ such that
$q\in B$ and $H(A,B)<\varepsilon$.\\

The relevance of the next result is that, in the class of continua having the Kelley's property, it gives a characterization of the locally connectedness of $X$ in 
terms of the hyperspaces $H(p,X)$ introduced here.

\begin{theorem}\label{locallyconnected}
Suppose that $X$ has the property of Kelly. Then the following are equivalent:
\begin{enumerate}
\item $X$ is locally connected,
\item $C(X)$ is locally connected,
\item there exist $p,q\in X$ with $p\neq q$ such that $HS(p,X)$ and $HS(q,X)$ are locally connected.
\end{enumerate}
\end{theorem}

\begin{proof}
By \cite[(1.92), p. 134]{Nadler(1978)}, conditions $1.$ and $2.$ are equivalent.\\
If $C(X)$ is locally connected, then $HS(p,X)$ is locally connected for each $p\in X$ because $HS(p,X)=\pi^X_p(C(X))$, so $2. \Rightarrow 3.$\\
In order to prove $3. \Rightarrow 1.$, suppose that there exist different points $p$ and $q$ in $X$ such that $HS(p,X)$ and $HS(q,X)$ are locally connected. Let $x\in X-\{p\}$ and let $U$ be an open subset in $X$ containing $x$. Take an open set $V$, such that $x\in V\subset (U\cap (X-\{p\}))$. Note that $\{x\}\in (\langle V \rangle\cap C(X))\subset C(X)-C(p,X)$. Since $C(X)-C(p,X)$ is homeomorphic to the open locally connected set $HS(p,X)-\{C^{X}_{p}\}$, there exists an open connected subset $\mathcal{V}$ of $C(X)-C(p,X)$ such that $\{x\}\in\mathcal{V}\subset(\langle V \rangle\cap C(X))$. From \cite[Lemma 2.2 and Theorem 3.5]{CharatonikIllanes(2006)}, $\bigcup \mathcal{V}$ is a connected open subset of $X$. Finally, observe that $x\in \bigcup\mathcal{V}\subset V\subset U$. This conclude that, $X$ is locally connected in each point of $X-\{p\}$. In a similar way, by switching the roles of $p$ and $q$, we have that $X$ is locally connected in $p$. Therefore, $X$ is locally connected.
\end{proof}

A \textit{free arc} in $X$ is an arc $\beta$ in $X$ such that $\beta-\{\textrm{end points}\}$ is open in $X$. A \textit{Hilbert cube} is any space homeomorphic to $\mathcal{Q}=\prod^{\infty}_{n=1}[0,1]_{n}$ with the product topology, where $[0,1]_{n} = [0,1]$ for each positive integer $n$.

\begin{theorem}
If $X$ is locally connected without free arcs, then for each $p\in X$, $HS(p,X)$ is homeomorphic to $\mathcal{Q}$. 
\end{theorem}
\begin{proof}
Since $X$ is locally connected without free arcs, $C(X)$ is homeomorphic to $\mathcal{Q}$ (see \cite[4.1]{Curtis(1978)}). On the other hand, from \cite[Theorem 2, p. 220]{Eberhart(1978)}, $C(p,X)$ is AR (absolute retract) and therefore contractible. Hence, $C(p,X)$ has the shape of a point in the sense of Borsuk (see \cite[5.5, p. 28]{Borsuk(1973)}). Thus, $C(X)-C(p,X)$ is homeomorphic to $\mathcal{Q}-\{q\}$ for some point $q\in \mathcal{Q}$ (see \cite[25.2]{Chapman(1976)}). Since $C(X)-C(p,X)$ is homeomorphic to $HS(p,X)-\{C^{X}_{p}\}$, we have that $\mathcal{Q}-\{q\}$ is homeomorphic to $HS(p,X)-\{C^{X}_{p}\}$. Therefore, $HS(p,X)$ is homeomorphic to $\mathcal{Q}$ (see \cite[2.23]{Chandler(1976)}).
\end{proof}

The following is an immediate consequence of the previous theorem and \cite[4.1]{Curtis(1978)}.

\begin{corollary}
If $X$ is homeomorphic to $\mathcal{Q}$, then for each $p\in X$, $HS(p,X)$ is homeomorphic to $C(X)$.
\end{corollary}

\section{Aposyndesis}

Let $x,y\in X$ with $x\neq y$, $X$ is \textit{aposyndetic at $x$ with respect to $y$} provided there exists a continuum $M$ such that $x\in\textrm{int}(M)$ and $y\in X-M$. If for each $y\in X-\{x\}$, $X$ is aposyndetic at $x$ with respect to $y$, then $X$ is \textit{aposyndetic at $x$}. We will say that $X$ is \textit{aposyndetic} if it is aposyndetic in each one of its points. A continuum $X$ is \textit{semi--aposyndetic} if for each pair of distinct elements of $X$, $x$ and $y$, $X$ is aposyndetic at $x$ with respect to $y$ or $X$ is aposyndetic at $y$ with respect to $x$. It is well known that $C(X)$ is an aposyndetic continuum (see \cite[Corollary 5.2]{Macias(2001)}).

\begin{theorem}\label{semi-aposyndetic}
If $p\in X$ is such that $HS(p,X)$ is locally connected in $C^{X}_{p}$, then $HS(p,X)$ is aposyndetic in $C^{X}_{p}$ and for each $[A],[B]\in HS(p,X)-\{C^{X}_{p}\}$ with $[A]\neq [B]$, $HS(p,X)$ is aposyndetic at $[A]$ with respect to $[B]$.  
\end{theorem}

\begin{proof}
Let $[A],[B]\in HS(p,X)-\{C^{X}_{p}\}$. In this case $A,B\in C(X)-C(p,X)$ and $A\neq B$. Since $C(X)$ is aposyndetic, there exists $\mathcal{K}\in C(C(X))$ such that $A\in \textrm{int}(\mathcal{K})\subset \mathcal{K}\subset C(X)-\{B\}$. \\
If $\mathcal{K}\cap C(p,X)=\emptyset$, then it is clear that $$[A]\in \textrm{int}(\pi^{X}_{p}(\mathcal{K}))\subset \pi^{X}_{p}(\mathcal{K})\subset HS(p,X)-\{[B]\},$$ and $\pi^{X}_{p}(\mathcal{K})$ is a continuum. \\
Now, suppose that $\mathcal{K}\cap C(p,X)\neq \emptyset$. Let $\mathfrak{W}$ be an open connected subset of $HS(p,X)-\{[A],[B]\}$ containing $C^{X}_{p}$ such that $[A],[B]\notin \textrm{cl}(\mathfrak{W})$. Thus, $\mathcal{W}=(\pi^{X}_{p})^{-1}(\mathfrak{W})$ is an open connected subspace of $C(X)$ and $C(p,X)\subset \mathcal{W}$. Therefore, we have that $\textrm{cl}(\mathcal{W})\cup \mathcal{K}\in C(C(X))$, $C(p,X)\subset \textrm{int}\left( \textrm{cl}(\mathcal{W})\cup \mathcal{K}\right)$ and $A\in \textrm{int}(\textrm{cl}(\mathcal{W})\cup \mathcal{K})\subset \textrm{cl}(\mathcal{W})\cup \mathcal{K}\subset C(X)-\{B\}$. The continuum $\pi^{X}_{p}(\textrm{cl}(\mathcal{W})\cup \mathcal{K})$ satisfies that $$[A]\in \textrm{int}(\pi^{X}_{p}\left( \textrm{cl}(\mathcal{W})\cup \mathcal{K})\right)\subset \pi^{X}_{p}\left( \textrm{cl}(\mathcal{W})\cup \mathcal{K}\right) \subset HS(p,X)-\{[B]\}.$$

Finally, since $HS(p,X)$ is locally connected in $C^{X}_{p}$, then $HS(p,X)$ aposyndetic in $C^{X}_{p}$.
\end{proof}

As an immediate consequence of the previous theorem we get the following.

\begin{corollary}
If $p\in X$ is such that $HS(p,X)$ is locally connected in $C^{X}_{p}$, then $HS(p,X)$ is semi--aposyndetic.
\end{corollary}


Let $p\in X$. A continuum $A\in C(p,X)$ is called {\it terminal at} $p$ if for each $B\in C(p,X)$ we have that either $A\subset B$ or $B\subset A$. We say that $A$ is {\it terminal} provided it is terminal at $a$ for every $a\in A$.

It is well known that $A\in C(p,X)$ is terminal at $p$ if and only if $A$ is a cut point of $C(p,X)$ (see \cite[Lemma 3.7, p. 262]{Pellicer(2003)}).  We can get also, as a particular case of \cite[Lemma 3.9, p. 263]{Pellicer(2003)}, that $C(p,X)$ is an arc if and only if for each $A,B\in C(p,X)$ holds $A\subset B$ or $B\subset A$. According to this we will say that $p$ is a \textit{terminal point of $X$} when $C(p,X)$ is an arc.

\begin{theorem}\label{aposyndeticterminal}
If $p\in X$ is such that $HS(p,X)$ is locally connected in $C^{X}_{p}$, then $HS(p,X)$ is aposyndetic at $[A]\in HS(p,X)-\{C^{X}_{p}\}$ with respect to $C^{X}_{p}$ for each terminal point $A$ of $C(X)$.
\end{theorem}

\begin{proof}
Since $[A]\neq C^{X}_{p}$ and $C(X)$ is aposyndetic, for each $B\in C(p,X)$ there exist a subcontinuum $M_{B}$ of $C(X)$ such that $A\in \textrm{int}(M_{B})$ and $B\subset C(X)-M_{B}$. From the compactness of $C(p,X)$ there exists a finite subset $\{B_{1},\cdots,B_{m}\}$ of $C(p,X)$ such that $C(p,X)\subset \displaystyle\bigcup^{m}_{i=1}(C(X)-B_{i})$. Since $A$ is a terminal point of $C(X)$ we have that $A\in \textrm{int} \left(\displaystyle\bigcap^{m}_{i=1}B_{i}\right)$ and $\displaystyle\bigcap^{m}_{i=1}B_{i}\in C(C(X))$. Therefore, $[A]\in \textrm{int}\left(\pi^{X}_{p}\left( \displaystyle\bigcap^{m}_{i=1}B_{i}\right)\right)$ and $\pi^{X}_{p}\left( \displaystyle\bigcap^{m}_{i=1}B_{i}\right)\subset HS(p,X)-\{C^{X}_{p}\}$.  
\end{proof}

Theorems \ref{semi-aposyndetic} and \ref{aposyndeticterminal} lead to the following.

\begin{corollary}\label{aposyndesis}
If $p\in X$ is such that $HS(p,X)$ is locally connected in $C^{X}_{p}$, then $HS(p,X)$ is aposyndetic at $C^{X}_{p}$ and in each point $[A]\in HS(p,X)$ such that $A$ is a terminal point of $C(X)$.
\end{corollary}

\section{Colocally connectedness}

A space $Y$ is \textit{colocally connected in $y\in Y$} provided by for each open subset $U$ of $Y$ containing $y$ there exists an open subset $V$ such that $y\in V\subset U$ and $Y-V$ is connected. $Y$ is \textit{colocally connected} if is colocally connected in each one of its points. 

It is well known that $C(X)$ is a colocally connected continuum (see \cite[Theorem 5.1]{Macias(2001)}).

\begin{theorem}\label{colocally}
If $X$ is colocally connected in $p\in X$, then $HS(p,X)$ is colocally connected.
\end{theorem}

\begin{proof}
Let $[A]\in HS(p,X)$. In a first case, suppose that $[A]\neq C^{X}_{p}$. Let $\mathfrak{U}$ be an open subset of $HS(p,X)$ containing $[A]$. Without loss of generality, suppose that $C^{X}_{p}\notin \mathfrak{U}$. Therefore, $A\in (\pi^{X}_{p})^{-1}(\mathfrak{U})\subset C(X)-C(p,X)$. Since $C(X)$ is colocally connected, we can consider an open subset $\mathcal{V}$ of $C(X)$ such that $A\in \mathcal{V}\subset (\pi^{X}_{p})^{-1}(\mathfrak{U})$ and $C(X)-\mathcal{V}$ is connected. Thus, $\pi^{X}_{p}(\mathcal{V})$ is an open subset of $HS(p,X)$ such that $[A]\in \pi^{X}_{p}(\mathcal{V})\subset \mathfrak{U}$ and $HS(p,X)-\pi^{X}_{p}(\mathcal{V})=\pi^{X}_{p}(C(X)-\mathcal{V})$ is connected. We conclude that $HS(p,X)$ is colocally connected in $[A]$.\\
It remains to analyze the case when $[A]=C^{X}_{p}$. In order to prove that $HS(p,X)$ is locally connected in $C^{X}_{p}$, let $\mathfrak{W}$ be an open subset of $HS(p,X)$ containing $C^{X}_{p}$. Then, $(\pi^{X}_{p})^{-1}(\mathfrak{W})$ is an open subset of $C(X)$ and $C(p,X)\subset (\pi^{X}_{p})^{-1}(\mathfrak{W})$. Since $\{p\}\in C(p,X)$ there exists $\varepsilon>0$ such that $B_{\varepsilon}(\{p\})\subset (\pi^{X}_{p})^{-1}(\mathfrak{W})$. By hypothesis, there exists an open subset $W$ of $X$ such that $p\in W\subset B_{\frac{\varepsilon}{2}}(p)$ and $X-V$ is connected. Let $\delta=\sup\{d(p,w):w\in W\}$. Note that $0<\delta<\varepsilon$. To finish the proof it is enough to show that $C(X)-N(\delta, C(p,X))$ is connected. To see this, observe that $F_{1}(X-V)\subset C(X)-N(\delta, C(p,X))$ and $F_{1}(X-V)$ is connected. Our task now consist in joining each point of $C(X)-N(\delta, C(p,X))$ which some point of $F_{1}(X-V)$. Now, given $A\in C(X)-N(\delta, C(p,X))$, it follows that $H(A,\{p\})\geq\delta$, then there exists $a\in A$ such that $d(a,p)\leq \delta$. Thus, $a\in X-V$. Let $\alpha:[0,1]\to C(X)$ be an ordered arc from $\{a\}$ to $A$. It is clear that $\alpha([0,1])$ is a connected subset of $C(X)-N(\delta,C(p,X))$ containing $A$ and $\alpha([0,1])\cap F_{1}(X-V)=\{\{a\}\}$.  
\end{proof}

Since colocal connectedness implies aposyndesis, we have the following result.

\begin{corollary}\label{aposyndetic}
If $X$ is colocally connected in $p\in X$, then $HS(p,X)$ is aposyndetic.
\end{corollary}

We can proof an even finer result. To stay it, we need before to recall that $X$ is \textit{finitely aposyndetic} if for each finite subset $F$ of $X$ and point $x$ of $X$ not in $F$ , there exists a subcontinuum $W$ of $X$ such that 
$x\in \textrm{int}(W)\subset W\subset X\setminus F$.

\begin{corollary}
If $X$ is colocally connected in $p\in X$, then $HS(p,X)$ is finitely aposyndetic.
\end{corollary}

\begin{proof}
By Theorem \ref{unicoherent} and Corollary \ref{aposyndetic}, $HS(p,X)$ is unicoherent and aposyndetic. The result now follows directly from \cite[Corollary 1]{Bennett(1971)}.
\end{proof}

\begin{theorem}\label{characterization}
Suppose that $X$ has the property of Kelley. Then the following are equivalent:
\begin{enumerate}
\item $X$ is an arc or a simple closed curve.
\item There exist $p,q\in X$, $p\neq q$, such that $X$ is colocally connected in $p$ and $q$ and the hyperspaces $HS(p,X)$ and $HS(q,X)$ can be embedded in $\mathbb{R}^{2}$.
\end{enumerate}
\end{theorem}

\begin{proof}
In order to prove $2.\Rightarrow 1.$, we assume that $p$ and $q$ satisfies $2$. By Theorem \ref{propertyb}, $HS(p,X)$ and $HS(q,X)$ have property b). Thus, $\mathbb{R}^{2}-HS(p,X)$ and $\mathbb{R}^{2}-HS(q,X)$ are connected (see \cite[VI 13, p. 100]{HurewiczWallman(1948)}). By Corollary \ref{aposyndetic}, $HS(p,X)$ and $HS(q,X)$ are aposyndetic. Therefore, $HS(p,X)$ and $HS(q,X)$ are locally connected (see \cite[Theorem 1, p. 139]{Jones(1952)}). So, by Theorem \ref{locallyconnected}, $X$ and $C(X)$ are locally connected. By using Remark \ref{ncells}, we have that $X$ does not contain $3$--ods. By \cite[31.12, p. 246]{IllanesNadler(1999)}, we conclude that $X$ is an arc or a simple closed curve.

To show $1.\Rightarrow 2.$, if $X$ is an arc with end points $p$ and $q$, it is clear that $X$ is colocally connected in $p$ and $q$, and if $X$ is a simple closed curve then $X$ is colocally connected. The result follows now from examples \ref{arco} and \ref{circle}.
\end{proof}

We can rewrite the last result to give some characterizations as follows. 

\begin{corollary}
Suppose that $X$ has the property of Kelley. Then,
\begin{enumerate}
\item $X$ is an arc if and only if there exists a unique pair of points $p,q\in X$ such that $X$ is colocally connected in $p$ and $q$ and the hyperspaces $HS(p,X)$ and $HS(q,X)$ can be embedded in $\mathbb{R}^{2}$.

\item $X$ is a simple closed curve if and only if $X$ is colocally connected and there exist $p,q\in X$, $p\neq q$, such that $HS(p,X)$ and $HS(q,X)$ can be embedded in $\mathbb{R}^{2}$.
\end{enumerate}
\end{corollary}

\section{Open questions}

To finish this paper, we write down some open question about the hyperspaces $HS(p,X)$.

In relation with the Example \ref{finales}, we have the following.

\begin{question}
What are all the trees, $X$, such that if $p\in E(X)$, $HS(p,X)$ is homeomorphic to $C(X)$?
\end{question}

Concerning Corollary \ref{aposyndesis}, we have the following question.

\begin{question}
If $X$ is a continuum and $p\in X$, is $HS(p,X)$ aposyndetic (or semi--aposyndetic)?
\end{question}

In \cite[Theorem 3.3, p.112]{Escobedo(2004)} and \cite[Corollary 3.5, p. 128]{Macias(2004)} are presented 
results about the topological structure of $HS_{n}(X)$ in the case when $X$ is hereditarily indecomposable. 
In this sense, we present the following problem:

\begin{problem}
Is it possible to characterize some classes of continua such as dendrites, dendroids, indecomposables, 
hereditarily indecomposables, in terms of some properties of the hyperspaces $HS(p,X)$?
\end{problem}

We finish the paper with the following question.

\begin{question}
Can the Kelly condition in any of the theorems \ref{locallyconnected} or  \ref{characterization} put it aside?
\end{question}

\section{Acknowledgement}
The authors wish to thank to Rosemberg Toal\'a, Eli Vanney Roblero and Sergio Guzm\'an for the fruitful discussions about the topic of this paper in the Seminario de Topolog\'ia, Geometr\'ia y sus Aplicaciones FCFM--UNACH.


\end{document}